\pdfoutput=1
\documentclass[11pt,a4paper]{article}
\usepackage[utf8]{inputenc}
\usepackage{amsmath,amssymb,amsthm}
\usepackage{booktabs}
\usepackage{longtable}
\usepackage[table,dvipsnames]{xcolor}
\usepackage{graphicx}
\usepackage{tikz}
\usetikzlibrary{shapes.geometric,arrows.meta,positioning,fit,backgrounds,calc}
\usepackage{geometry}
\usepackage[colorlinks=true,linkcolor=RoyalBlue,citecolor=BrickRed,urlcolor=RoyalBlue]{hyperref}
\usepackage{caption}
\newtheorem{theorem}{Theorem}
\newtheorem{proposition}{Proposition}

\theoremstyle{definition}
\newtheorem{definition}{Definition}

\newcommand{\ZZ}{\mathbb{Z}}
\newcommand{\NN}{\mathbb{N}}

\newcommand{\su}[1]{SU(#1)}
\newcommand{\so}[1]{SO(#1)}
\newcommand{\sy}[1]{Sp(#1)}

\definecolor{hdrblue}{RGB}{31,78,121}
\definecolor{softgrey}{RGB}{238,242,247}
\definecolor{warmred}{RGB}{178,34,34}

\newsavebox{\fichabox}
\newenvironment{ficha}[1]{%
  \par\medskip\begin{lrbox}{\fichabox}\begin{minipage}{0.93\textwidth}%
  {\normalsize\bfseries\color{hdrblue}#1}\par\medskip\small%
}{%
  \end{minipage}\end{lrbox}%
  \begin{center}\setlength{\fboxsep}{9pt}%
  \fcolorbox{hdrblue}{softgrey}{\usebox{\fichabox}}\end{center}\par\medskip%
}

\title{\textbf{\Large An Affine Semigroup from Orbifold Boundary Conditions}\\[7pt]
\large\textnormal{\itshape Cut, phylogenetic and hierarchical models in the unit-weight sector;\\ weighted orbifold configurations beyond them}\\[8pt]
\normalsize\textnormal{Part IX-B --- companion to \emph{The Alphabet of Orbifold Boundary
Conditions}}}
\author{Carles Mar\'in\thanks{Independent researcher, \texttt{karlesmarin@gmail.com}.
The computations were carried out and cross-checked with Claude (Anthropic) as an AI research
assistant against a common machine-verifiable ground truth; every number in this note regenerates
from the ancillary scripts, each of which writes its own receipt.}}
\date{\today}

\begin{document}
\maketitle

\begin{abstract}
\noindent
The equivalence classes of boundary conditions of a gauge theory on a two-dimensional orbifold are
the fibres of a marginal map, and are naturally indexed by the elements of an affine semigroup: one
generator per label of the orbifold's alphabet, graded by weight, embedded by its local data at the
fixed points. This note identifies that semigroup. When
the alphabet has no weights the configuration is one with a name and a literature: over $\ZZ_2$
it is the cut configuration of an explicit graph in the sense of Sturmfels and Sullivant, so that
its semigroup ring is the cut ring and its toric ideal the \emph{cut ideal} --- the four-cycle for
$T^2/\ZZ_2$ and the wheel $W_4$ for $S^1/\ZZ_2\times S^1/\ZZ_2$, verified here as an equality of
point configurations and not as a combinatorial resemblance --- over $\ZZ_m$, when the cone orders
are equal, it is the configuration of the group-based phylogenetic model on a claw tree, and when
they are not it is a mixed-order variant of that model which the literature does not seem to carry;
for the higher products it is the binary hierarchical model of the boundary complex of a
cross-polytope. Everything those frameworks already prove about our
cases is attributed here, including the fact that the ring of $S^1/\ZZ_2\times S^1/\ZZ_2$ is a row
of a 2008 table --- codimension, degree, minimal generators, normality --- every invariant of which
our machinery reproduced without knowing that table existed. What none of the three
covers is the alphabet \emph{with} weights, which arise from two distinct mechanisms: induction to
irreducibles of dimension greater than one of a space group that is not abelian, and recombination of
conjugate pairs when the defining space is real or quaternionic. That weighted sector --- which is
the one the physics needs, since the weighted labels are the non-diagonal boundary conditions ---
lies adjacent to --- but is not identified here with --- the non-abelian direction Sturmfels
and Sullivant raised in 2005 with a single computed data point, and is where this note's own contributions sit:
the gluing trees of the weighted alphabets and of the orthogonal and symplectic columns, and the
classification of the group-based model on the tripod, which for every finite abelian group is a
complete intersection exactly when the group has order at most three. The first group beyond
$\ZZ_3$ is also what separates local complete-intersection behaviour from the global toric
complete-intersection property: the $\ZZ_4$ tripod is a complete intersection on the Zariski-open
set the phylogenetics literature works in, and is not one globally.
\end{abstract}

\begin{ficha}{The note on one card}
\noindent
\textbf{The object.} One affine semigroup per orbifold and gauge group: one generator per label,
graded by weight, embedded by its local data at the fixed points (Definition~\ref{def:S}). The
alphabet is its minimal generating set in every case listed, which is a condition on the alphabet
once the weights differ and is verified rather than assumed.

\smallskip\noindent
\textbf{What it already is} (\S\ref{sec:dictionary}), and therefore what is \emph{not} claimed here:
over $\ZZ_2$ the cut configuration of a graph, whose toric
ideal is that graph's cut ideal~\cite{SScuts}; over $\ZZ_m$ with unit weights \emph{and equal cone
orders} the configuration of a group-based model~\cite{SSphylo}; for $(S^1/\ZZ_2)^k$ that of a
binary hierarchical model. The ring of the product orbifold is a row of Table 1
of~\cite{SScuts} --- codimension $7$, degree $64$, $8+8$ minimal generators, normal and
Cohen--Macaulay --- and that the degree-$2$ moves cannot suffice there is their Corollary 3.3(2)
together with Engstr\"om's theorem~\cite{Engstrom}.

\smallskip\noindent
\textbf{What is established here.}
\begin{itemize}\itemsep1pt \parskip0pt \topsep2pt
\item the identification itself, as an equality of point configurations (\S\ref{sec:dictionary});
\item thirteen gluing trees --- the unitary, orthogonal and symplectic realisations of every case ---
read as what \cite{SSphylo} says they are, exceptions (\S\ref{sec:gluings});
\item the group-based model on the tripod, for \emph{every} finite abelian group: a complete
intersection exactly when $|G|\le3$ (\S\ref{sec:tripod});
\item $(S^1/\ZZ_2)^k$ is not a complete intersection for any $k\ge3$, and is not a cut ideal either,
for a reason that is a count of edges (\S\ref{sec:allk});
\item the exceptional case as a list of invariants rather than a negation: a \emph{normal,
Cohen--Macaulay} toric ring of normalised volume $64$ that is \emph{not} a complete
intersection and has Cohen--Macaulay type $9$ (Table~\ref{tab:invariants},
Proposition~\ref{prop:type}); and that its minimal Markov basis is \emph{unique} up to
signs --- sixteen indispensable binomials in two orbits of eight, and the quartics exist
because of eight pairs of boundary conditions the quadrics cannot exchange
(Proposition~\ref{prop:eight});
\item and the weighted alphabets, which no framework above covers (\S\ref{sec:weights}).
\end{itemize}

\smallskip\noindent
\textbf{The distinction that matters most.} Casanellas, Fern\'andez-S\'anchez and
Micha{\l}ek~\cite{CFM} prove that these varieties are complete intersections \emph{in a Zariski open
set} --- their Corollary 3.6 says the binomials involved are a lattice basis, and a lattice-basis
ideal has as many generators as the codimension for any lattice at all. The question here is the
global one, about the toric ideal itself, and the two are not the same statement
(\S\ref{sec:localglobal}).
\end{ficha}

\section{The semigroup}\label{sec:semigroup}

The object of this note comes from gauge theory, and one paragraph is enough to say where; nothing
below uses any physics.

A gauge theory on a two-dimensional orbifold admits a discrete set of boundary conditions,
classified up to gauge equivalence. In the companion note~\cite{IXA} the classification is put in
the following form. The orbifold determines a finite \emph{alphabet}: a set of labels, each carrying
a positive integer \emph{weight} --- the dimension of an irreducible representation of the
orbifold's space group --- and each carrying a tuple of \emph{local data}: at each cone point, the
multiplicities with which the eigenvalues of the local rotation occur. A boundary condition of rank
$N$ is a multiset of labels of total weight $N$; its class is the sum of its labels' local data.
The division of labour between the two notes is clean enough to be worth stating: what is settled
there, by M\"obius inversion and Clifford theory, is which generators exist; what is settled here,
by toric and Markov-basis arguments, is which relations hold among them.

\begin{definition}\label{def:S}
Let $L=\{\ell_1,\dots,\ell_n\}$ be the alphabet, $w_i\in\ZZ_{>0}$ the weight of $\ell_i$ and
$v_i\in\NN^{c}$ its local data. Put $a_i=(w_i,v_i)\in\NN^{1+c}$ and let
$S=\langle a_1,\dots,a_n\rangle\subseteq\NN^{1+c}$ be the affine semigroup they generate. The number
of equivalence classes of rank $N$ is the number of elements of $S$ in degree $N$, so the
class-counting generating function is the Hilbert series $H(S,x)$.
\end{definition}

The lattice of relations $\mathcal L=\ker(\ZZ^n\to\ZZ^{1+c})$ has rank $n-\dim S$, and its elements
are the moves preserving every local datum.

\paragraph{Minimality.} The gluing criterion below is a statement about the \emph{minimal}
generating set, so $\{a_i\}$ must be one. For a weight-one alphabet that is immediate: a sum of two
generators has weight $2$ and no generator does. For a weighted alphabet it is a condition on the
alphabet rather than on the grading --- with weights $1,1,2$ three generators can be distinct and
still satisfy $a_3=a_1+a_2$ --- and it holds in each of the eleven configurations of
Table~\ref{tab:family} and in the three symplectic ones of Table~\ref{tab:gluings}, by exhaustive
search over the multisets of the right total weight: no
generator lies in the semigroup generated by the others. The search carries a decoy, the same
alphabet with the sum of two of its labels appended, which it reports as non-minimal and whose
redundant generator it names.

\begin{table}[h]
\centering\footnotesize
\begin{tabular}{llrrrl}
\toprule
\rowcolor{hdrblue}
\textcolor{white}{orbifold} & \textcolor{white}{group} & \textcolor{white}{$n$} &
\textcolor{white}{$\dim S$} & \textcolor{white}{rank $\mathcal L$} &
\textcolor{white}{$H(S,x)$}\\
\midrule
$S^1/\ZZ_2$ & $\su N$ & $4$ & $3$ & $1$ & $(1-x^2)/(1-x)^4$\\
\rowcolor{softgrey}
$T^2/\ZZ_2$ & $\su N$ & $8$ & $5$ & $3$ & $(1-x^2)^3/(1-x)^8$\\
$T^2/\ZZ_3$ & $\su N$ & $9$ & $7$ & $2$ & $(1-x^3)^2/(1-x)^9$\\
\rowcolor{softgrey}
$T^2/\ZZ_4$ & $\su N$ & $10$ & $8$ & $2$ & $(1-x^4)^2/\big[(1-x)^8(1-x^2)^2\big]$\\
$T^2/\ZZ_6$ & $\su N$ & $11$ & $9$ & $2$ &
$(1-x^6)^2/\big[(1-x)^6(1-x^2)^3(1-x^3)^2\big]$\\
\rowcolor{softgrey}
$T^2/\ZZ_4$ & $\su N$, weight $1$ & $8$ & $7$ & $1$ & $(1-x^4)/(1-x)^8$\\
$T^2/\ZZ_6$ & $\su N$, weight $1$ & $6$ & $6$ & $0$ & $1/(1-x)^6$\\
\rowcolor{softgrey}
$T^2/\ZZ_3$ & $\so N$ & $5$ & $4$ & $1$ & $(1-x^6)/\big[(1-x)(1-x^2)^4\big]$\\
$T^2/\ZZ_4$ & $\so N$ & $8$ & $6$ & $2$ & $(1-x^4)^2/\big[(1-x)^4(1-x^2)^4\big]$\\
\rowcolor{softgrey}
$T^2/\ZZ_6$ & $\so N$ & $8$ & $6$ & $2$ &
$(1-x^6)^2/\big[(1-x)^2(1-x^2)^3(1-x^3)^2(1-x^4)\big]$\\
\midrule
$S^1/\ZZ_2\times S^1/\ZZ_2$ & $\su N$ & $16$ & $9$ & $7$ &
$\dfrac{1+7x+20x^2+28x^3+7x^4+x^5}{(1-x)^9}$\\
\bottomrule
\end{tabular}
\caption{The family. The class counts of the first ten rows are in print in the physics
literature~\cite{HHK,KM,TI2,TI3}; the last row is the one the companion note~\cite{IXA} derives,
and is the only row of the table that is not a quotient by a cyclic rotation. The rows with a weight-$1$
alphabet are the ones \S\ref{sec:dictionary} identifies with objects already studied in
commutative algebra; the rows whose alphabet carries weights $2$, $3$ and $4$ are
\S\ref{sec:weights}.}
\label{tab:family}
\end{table}

\section{What this semigroup already is}\label{sec:dictionary}

Three identifications, in increasing order of how much they cost to state.

\paragraph{Over $\ZZ_2$: the cut configuration.} Sturmfels and Sullivant~\cite{SScuts} attach to a
graph $G$ the toric ideal whose coordinates are indexed by the cuts of $G$ --- the kernel of the map
from the polynomial ring to the semigroup ring of the cut configuration --- and its polytope
$\mathrm{Cut}^\square(G)$ is the convex hull of the cut vectors. For each of the three orbifolds of
Table~\ref{tab:family} whose cone orders are all $2$ the alphabet \emph{is} that vertex set, so the
semigroup here is their cut semigroup and its toric ideal is their cut ideal:

\begin{itemize}\itemsep2pt
\item $S^1/\ZZ_2$: a label is a parity under each of the two reflections, and those four sign
vectors are the four cuts of the path $P_3$;
\item $T^2/\ZZ_2$: a label is a sign at each of the four cone points with product $+1$, and those
eight sign vectors are exactly the eight cuts of the four-cycle $C_4$;
\item $S^1/\ZZ_2\times S^1/\ZZ_2$: a label is a sign under each of the four reflections
$P_0,P_1,Q_0,Q_1$, and the local datum is $N$, the four single counts and the four correlations
$\langle P_iQ_j\rangle$ --- which are the eight edges of a graph on five vertices, a hub joined to
each reflection and the four-cycle $P_0Q_0P_1Q_1$. That graph is the wheel $W_4$, the sixteen labels
are exactly its sixteen cuts, and the two point configurations are affinely isomorphic in both
directions.
\end{itemize}

\noindent
This is an equality of point configurations, checked in exact arithmetic, and not the observation
that two polytopes have the same face lattice.

\paragraph{Over $\ZZ_m$ with unit weights: a group-based model, and then not quite.} A label of
$T^2/\ZZ_m$ whose weight is $1$ is one root of unity per cone point with product the identity, and
the local datum is the indicator of each. When all the cone points have the same order that is, in
Fourier coordinates, the group-based phylogenetic model of $\ZZ_m$ on the claw tree $K_{1,l}$ with
$l$ the number of cone points~\cite{SSphylo}, whose codimension for a group of order $g$ is
$g^{\,l-1}-1-l(g-1)$; and it reproduces the ranks of Table~\ref{tab:family}, $3^2-1-3\cdot2=2$ for
$T^2/\ZZ_3$ and $2^3-1-4\cdot1=3$ for $T^2/\ZZ_2$.

The other two orbifolds are not that model, and the alphabet sizes say so at once: the weight-one
alphabet of $T^2/\ZZ_4$ has $8$ labels where a $\ZZ_4$ tripod would have $4^2=16$, and that of
$T^2/\ZZ_6$ has $6$ where a $\ZZ_6$ tripod would have $36$. The reason is that their cone orders are
mixed --- $(4,4,2)$ and $(6,3,2)$ --- so each leaf takes values in a different subgroup. What they
are is the configuration of
\begin{equation}\label{eq:mixed}
\boxed{\;\mathcal A(G;H_1,\dots,H_l)\;=\;\Big\{(g_1,\dots,g_l)\ :\ g_i\in H_i,\ \textstyle\sum_i
g_i=0\Big\},\;}
\end{equation}
with $H_i\le G$ the subgroup of order $m_i$ at the $i$-th cone point, embedded by the indicators as
before. Taking $H_i=G$ for every $i$ returns the group-based model; the orbifolds with mixed cone
orders return proper subgroups, one per leaf, and the orbifold decides which. We have not found
\eqref{eq:mixed} treated in the group-based literature, where a single group runs over all the
edges.

\paragraph{For the higher products: a binary hierarchical model.} On $(S^1/\ZZ_2)^k$ the ground set
is the $2k$ reflections, in $k$ antipodal pairs, and a fixed point is a choice of one reflection per
circle. So the marginals are indexed by the $2^k$ \emph{transversals}, which are the facets of the
boundary complex of the $k$-dimensional cross-polytope. Counting states and parameters is then
immediate and replaces the Kronecker argument of~\cite{IXA}:
\begin{equation}\label{eq:kcount}
\boxed{\;n=2^{2k}=4^k,\qquad \dim S=\#\{\text{faces}\}=\sum_{j}\binom{k}{j}2^{\,j}=3^k.\;}
\end{equation}
At $k=2$ the complex is the four-cycle and the model is the binary graph model of $C_4$, which
by~\cite[\S4]{SScuts} corresponds to the cut ideal of its suspension --- the wheel $W_4$ again, so
the two identifications agree where they overlap.

\section{What the identification buys, and what it costs}\label{sec:known}

It costs the binary cases, and they should be attributed plainly.

\paragraph{$T^2/\ZZ_2$ is their Example 1.2.} \cite{SScuts} compute $I_{C_4}$ and conclude that the
variety \emph{is a complete intersection of three quadrics}. That is the content of the $T^2/\ZZ_2$
row of \S\ref{sec:gluings}, and Table 1 of the same paper records its codimension $3$ and degree
$8$ --- the latter being $\sum h^*_i$ for the $h^*$-vector $(1,3,3,1)$ measured here.

\paragraph{The product orbifold is a row of their Table 1.} Their $\widehat{G}$ is the suspension of
$G$ over a point, so $\widehat{C_4}$ is $W_4$, and the row reads: $8$ minimal generators of degree
$2$, $8$ of degree $4$, codimension $7$, degree $64$, normal, Cohen--Macaulay, not Gorenstein. The
computations of this note return the same numbers by routes that do not communicate --- a minimal
Markov basis of $8+8$ elements, a relation lattice of rank $7$, and an $h^*$-vector
$(1,7,20,28,7,1)$ summing to $64$ --- which is a strong external check of the machinery: the
control was set by someone else, seventeen years earlier, and every invariant came out equal. What
their row records is the \emph{ring}. The class counts are the Ehrhart series of that ring with
respect to $L$, which the row does not print and \S\ref{sec:exception} computes; what 2008 settles
is that the object carrying them had already been identified and measured.

\paragraph{And the degree-$4$ moves are a corollary of theirs.} \cite[Cor.~3.3(2)]{SScuts}: a graph
with a $K_4$ minor has a minimal generator of degree $4$ in its cut ideal. $W_4$ has one. In the
other direction, Engstr\"om~\cite{Engstrom} proved their Conjecture 3.5: the cut ideal is generated
by quadrics if and only if the graph is free of $K_4$ minors. So the fact that on the product
orbifold moves of degree $2$ leave components uncovered and moves of degree $4$ close them is that
theorem applied to the graph the orbifold defines --- and the corresponding fact for $T^2/\ZZ_2$, whose
graph $C_4$ is $K_4$-minor free, is the same theorem in the other direction. What the orbifold
supplies is the graph.

\section{Thirteen gluings, read as exceptions}\label{sec:gluings}

Recall the criterion. For a partition $T=T_1\sqcup T_2$ of the minimal generating set into non-empty
parts, $S$ is the \emph{gluing} of $S_1=\langle T_1\rangle$ and $S_2=\langle T_2\rangle$ by $\alpha$
if $\alpha\in S_1\cap S_2$, $\alpha\neq0$, and $G(S_1)\cap G(S_2)=\ZZ\alpha$. Fischer, Morris and
Shapiro~\cite{FMS} prove that an affine semigroup which is not free is a complete intersection if
and only if it is a gluing of two complete intersections; this generalises Delorme's rank-one
theorem~\cite{Delorme}. Assi, Garc\'ia-S\'anchez and Ojeda~\cite{AGO} give the Hilbert series:
\begin{equation}\label{eq:ago}
\boxed{\;H(S_1+_\alpha S_2,\,x)\;=\;(1-x^{\deg\alpha})\,H(S_1,x)\,H(S_2,x).\;}
\end{equation}

Applying it to Table~\ref{tab:family} is a search over partitions with an exact test on each: the
lattice intersection by Hermite normal form and membership by reachability in the graded semigroup.

\begin{table}[h]
\centering\small
\begin{tabular}{llrlll}
\toprule
\rowcolor{hdrblue}
\textcolor{white}{orbifold} & \textcolor{white}{group} & \textcolor{white}{$n$} &
\textcolor{white}{gluing degrees} & \textcolor{white}{numerator degrees} &
\textcolor{white}{status}\\
\midrule
$S^1/\ZZ_2$ & $\su N$ & $4$ & $[2]$ & $[2]$ & binary, \cite{SScuts}\\
\rowcolor{softgrey}
$T^2/\ZZ_2$ & $\su N$ & $8$ & $[2,2,2]$ & $[2,2,2]$ & their Ex.~1.2\\
$T^2/\ZZ_3$ & $\su N$ & $9$ & $[3,3]$ & $[3,3]$ & group-based, $\ZZ_3$ tripod\\
\rowcolor{softgrey}
$T^2/\ZZ_4$ & $\su N$, w.\ $1$ & $8$ & $[4]$ & $[4]$ & group-based, mixed orders\\
$T^2/\ZZ_6$ & $\su N$, w.\ $1$ & $6$ & $[\,]$ & $[\,]$ & free, $\NN^6$\\
\rowcolor{softgrey}
$T^2/\ZZ_4$ & $\su N$ & $10$ & $[4,4]$ & $[4,4]$ & \textbf{weighted}\\
$T^2/\ZZ_6$ & $\su N$ & $11$ & $[6,6]$ & $[6,6]$ & \textbf{weighted}\\
\rowcolor{softgrey}
$T^2/\ZZ_3$ & $\so N$ & $5$ & $[6]$ & $[6]$ & \textbf{weighted}\\
$T^2/\ZZ_4$ & $\so N$ & $8$ & $[4,4]$ & $[4,4]$ & \textbf{weighted}\\
\rowcolor{softgrey}
$T^2/\ZZ_6$ & $\so N$ & $8$ & $[6,6]$ & $[6,6]$ & \textbf{weighted}\\
\midrule
\rowcolor{softgrey}
$T^2/\ZZ_3$ & $\sy N$ & $5$ & $[3]$ & $[3]$ & quaternionic\\
$T^2/\ZZ_4$ & $\sy N$ & $8$ & $[4,4]$ & $[4,4]$ & \textbf{weighted}\\
\rowcolor{softgrey}
$T^2/\ZZ_6$ & $\sy N$ & $8$ & $[6,6]$ & $[6,6]$ & \textbf{weighted}\\
\bottomrule
\end{tabular}
\caption{Thirteen complete intersections. The two right-hand degree columns are computed by routes
that do not communicate: the gluing degrees from lattice geometry, the numerator degrees from the
Hilbert series. The last column says which framework the row belongs to; seven rows are weighted,
and no framework in \S\ref{sec:dictionary} covers those.}
\label{tab:gluings}
\end{table}

\paragraph{The same alphabet, three realisations.} The last three rows complete a comparison the
table would otherwise only half make. The three orbifolds $T^2/\ZZ_3$, $T^2/\ZZ_4$ and $T^2/\ZZ_6$
each have one complex alphabet, realised in three ways with three different weight profiles and
three different dimensions:

\begin{center}\small
\begin{tabular}{llll}
\toprule
\rowcolor{hdrblue}
\textcolor{white}{orbifold} & \textcolor{white}{over $\su N$} & \textcolor{white}{over $\so N$} &
\textcolor{white}{over $\sy N$}\\
\midrule
$T^2/\ZZ_3$ & $9{\times}1$, \ $[3,3]$ & $1{\times}1+4{\times}2$, \ $[6]$ &
$5{\times}1$, \ $[3]$\\
\rowcolor{softgrey}
$T^2/\ZZ_4$ & $8{\times}1+2{\times}2$, \ $[4,4]$ & $4{\times}1+4{\times}2$, \ $[4,4]$ &
$6{\times}1+2{\times}2$, \ $[4,4]$\\
$T^2/\ZZ_6$ & $6{\times}1+3{\times}2+2{\times}3$, \ $[6,6]$ &
$2{\times}1+3{\times}2+2{\times}3+1{\times}4$, \ $[6,6]$ &
$4{\times}1+2{\times}2+2{\times}3$, \ $[6,6]$\\
\bottomrule
\end{tabular}
\end{center}

\noindent
Every one of the nine is a complete intersection, and in eight of the nine \emph{every} gluing
degree is the order $m$ of the rotation --- although the \emph{number} of gluings changes with the
type, since the dimension does: $T^2/\ZZ_3$ has two gluings over $\su N$ and one over $\sy N$. The
single row whose degree is not $m$ is $T^2/\ZZ_3$ over $\so N$, and it is the only one of the nine
whose alphabet carries a weight-one label among heavier ones.

\paragraph{These are exceptions, and the literature says so.} \cite[\S7]{SSphylo} state it flatly:
\emph{most models in algebraic statistics, including the group-based evolutionary models treated
there, are not complete intersections}. Table~\ref{tab:gluings} should therefore be read as a list
of the cases that survive, not as evidence that the family is well behaved --- and
\S\ref{sec:tripod} shows how quickly the surviving cases run out.

\paragraph{The gluing element is the physics.} In each case the $\alpha$ realising a gluing is,
read back, one of the equivalence relations printed in the physics source: the identity saying that
a collection of labels may be replaced by another with the same local data.

\paragraph{A degree that is not the order of the orbifold.} The thirteen rows carry twenty-one
gluings between them, and twenty of the twenty-one have degree $m$. The exception is the single
gluing of $T^2/\ZZ_3$ over $\so N$, of degree $6=2m$, and that row is the only one whose alphabet
has a weight-one label among heavier ones. The degree of a gluing is the weight of
its gluing element, and $(1-x^m)$ is the special case that occurs when all weights are $1$.

\begin{figure}[h]
\centering
\includegraphics[width=0.88\textwidth]{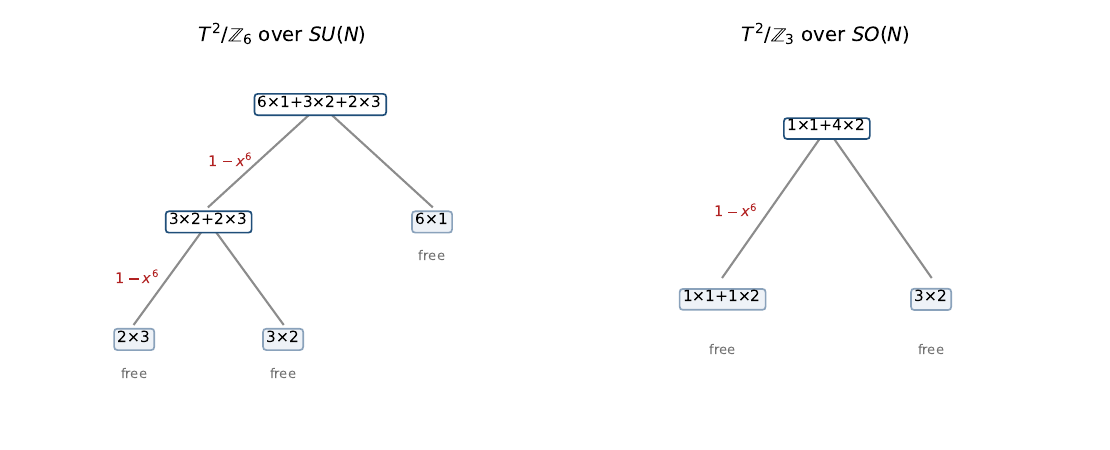}
\caption{Two gluing trees, as the search returns them. Each node is a set of generators, written by
its multiset of weights; each split is a gluing and contributes the factor on the branch
to~\eqref{eq:ago}; each leaf is a free semigroup. Both trees shown are weighted cases, so both are
outside the frameworks of \S\ref{sec:dictionary}.}
\label{fig:gluing}
\end{figure}

\section{The tripod, decided}\label{sec:tripod}

The smallest member of the group-based family is the claw tree with three leaves. For a cyclic group
$\ZZ_g$ its codimension is $(g-1)(g-2)$, and \cite{CFM} construct that many invariants of degree at
most $g$; whether the toric ideal is a complete intersection is a different question
(\S\ref{sec:localglobal}), and the gluing criterion answers it. One case is an orbifold of ours:
$g=3$ is $T^2/\ZZ_3$ over $\su N$.

\begin{proposition}\label{prop:tripod}
Let $G$ be a finite abelian group and let $S$ be the semigroup of the group-based model on the
tripod. Then $S$ is a complete intersection if and only if $|G|\le3$: it is free for $|G|=2$, a
complete intersection with two cubics for $|G|=3$, and not a complete intersection for $\ZZ_4$, for
$\ZZ_2\times\ZZ_2$, or for any abelian group of order at least $5$.
\end{proposition}

\begin{proof}
The labels are the triples of group elements summing to zero, so $n=|G|^2$, and the local datum is
the indicator of each entry, so $\dim S=3(|G|-1)+1=3|G|-2$ and the codimension is $(|G|-1)(|G|-2)$,
which is the value the phylogenetics literature prints for the tripod.

Every generator is an extreme ray, by the argument of Proposition~\ref{prop:allk}: it is a $0/1$
vector with one $1$ in each of three blocks, hence one of the vertices of a product of three
simplices, and a vertex stays extreme in the hull of any subset containing it. Nothing in that
argument uses the group law beyond the fact that the labels select a subset of those vertices,
so it holds unchanged for every finite abelian $G$ and not only for the cyclic ones ---
which is what lets the bound below be quantified over all of them. Theorem~\ref{thm:cor34}
therefore requires
\begin{equation}\label{eq:tripodbound}
\boxed{\;|G|^2\;\le\;2\dim S-2\;=\;6|G|-6,\qquad\text{that is}\qquad |G|^2-6|G|+6\le0,\;}
\end{equation}
which holds for $2\le|G|\le4$ and fails from $5$ on: writing $f(g)=g^2-6g+6$, one has $f(5)=1>0$ and
$f(g+1)-f(g)=2g-5>0$ for $g\ge3$, so $f$ is positive and increasing on the integers $g\ge5$.
The bound therefore settles every group of
order five or more; the trivial group is outside its reach for the other reason, $\dim S=1$ being
below the $d\ge2$ of Theorem~\ref{thm:cor34}, and its semigroup is $\NN$, free. That leaves four
groups open, of three orders: $\ZZ_2$, $\ZZ_3$, $\ZZ_4$ and
$\ZZ_2\times\ZZ_2$. Each of the four is decided by the gluing criterion, exhaustively --- the two of order $4$ over all
$2^{15}-1$ partitions of their sixteen generators --- with the results of Table~\ref{tab:tripod}.
\end{proof}

\begin{table}[h]
\centering\small
\begin{tabular}{lrrrrl}
\toprule
\rowcolor{hdrblue}
\textcolor{white}{$G$} & \textcolor{white}{$n=|G|^2$} & \textcolor{white}{$\dim S$} &
\textcolor{white}{codim} & \textcolor{white}{$2\dim S-2$} & \textcolor{white}{verdict}\\
\midrule
$\ZZ_2$ & $4$ & $4$ & $0$ & $6$ & free\\
\rowcolor{softgrey}
$\ZZ_3$ & $9$ & $7$ & $2$ & $12$ & complete intersection, degrees $[3,3]$\\
$\ZZ_4$ & $16$ & $10$ & $6$ & $18$ & not a complete intersection \ (swept)\\
\rowcolor{softgrey}
$\ZZ_2\times\ZZ_2$ & $16$ & $10$ & $6$ & $18$ & not a complete intersection \ (swept)\\
$|G|\ge5$ & $|G|^2$ & $3|G|-2$ & --- & $6|G|-6$ & not a complete intersection \
(\eqref{eq:tripodbound})\\
\bottomrule
\end{tabular}
\caption{The tripod, for every finite abelian group. The codimension is computed from the generator
matrix and agrees with $(|G|-1)(|G|-2)$, which is a control that could have failed. The bound is
silent exactly on the four small groups, which is why they are swept; and $\ZZ_2\times\ZZ_2$ is the
group of the Kimura 3-parameter model.}
\label{tab:tripod}
\end{table}

\noindent
We have looked for this classification in the group-based literature and not found it; what that
literature settles for the tripod is normality, and what it proves about complete intersections is
the local statement of \S\ref{sec:localglobal}. The reading matters as much as the statement. That
$T^2/\ZZ_3$ is a complete intersection is not a
property of tripods but of the order $3$: it stops at the next group, in both groups of order four,
and the extreme-ray bound closes everything above. It is the same shape as
\S\ref{sec:allk} --- a bound that decides all but the smallest cases, and an exhaustive sweep for
those --- and in both places the cases the bound cannot see are the interesting ones.

\section{Where the weights take it outside}\label{sec:weights}

A group-based model is abelian by construction: for the group-based phylogenetic models used
here, the Fourier transform of a finite abelian group is what puts the model in monomial
coordinates, in which it is toric. (Toric is of course far more general than that, and owes
nothing to abelian groups in general; the point is only how \emph{these} models arrive at it.)
Every framework of \S\ref{sec:dictionary} inherits that construction. Seven rows of
Table~\ref{tab:gluings} carry weights, and the weight is not a device: it is the dimension of an
irreducible, and the semigroup is graded by it. There is no Fourier transform to apply.

The weights have two distinct sources, and separating them is worth a sentence. Over $\su N$ they
come from \emph{induction}: a character of $\Lambda$ whose stabiliser is a proper subgroup of
$\ZZ_m$ induces an irreducible of dimension $>1$ of the space group $\Lambda\rtimes\ZZ_m$, which is
not abelian~\cite{IXA}. Over $\so N$ and $\sy N$ they come, in addition, from
\emph{recombination}: a complex irreducible that is not self-conjugate pairs with its conjugate into
a single real irreducible of twice the dimension, so weights appear even where the complex alphabet
had none. The two mechanisms are visible side by side in $T^2/\ZZ_3$: over $\su N$ every label has
weight $1$, and over $\so N$ the alphabet is $1{\times}1+4{\times}2$ with every weight $2$ produced
by recombination and none by induction.

\paragraph{One construction, four regimes.} Those two mechanisms and the mixed cone orders
of~\eqref{eq:mixed} are not three unrelated escapes from the dictionary. They are what one
construction produces as its inputs vary. The construction is the one of~\cite{IXA}: take the
irreducible representations of the space group whose restriction to the translations has non-trivial
stabiliser, weight each by its dimension, record at each cone point the multiset of eigenvalues of
the local rotation --- and then realify according to the type of the defining space. It has three
inputs, and each moves the configuration into a different regime:

\begin{center}\small
\begin{tabular}{lllll}
\toprule
\rowcolor{hdrblue}
\textcolor{white}{cone orders} & \textcolor{white}{stabilisers} & \textcolor{white}{type} &
\textcolor{white}{configuration} & \textcolor{white}{name}\\
\midrule
equal & all full & complex & flows on a claw tree & group-based model~\cite{SSphylo}\\
\rowcolor{softgrey}
equal, all $2$ & all full & complex & cuts of a graph & cut configuration~\cite{SScuts}\\
mixed & all full & complex & flows with $g_i\in H_i$ & \eqref{eq:mixed}, not located\\
\rowcolor{softgrey}
any & some proper & complex & \emph{weighted} flows & not located\\
any & any & real, quaternionic & weighted, recombined & not located\\
\bottomrule
\end{tabular}
\end{center}

\noindent
In that reading the general object is a \emph{weighted flow configuration}: the datum of a label at
the $i$-th cone point is no longer one element of $\widehat{H_i}$ but a multiset of $w$ of them, $w$
the weight, and the constraint that made the flows close is inherited from the representation rather
than imposed. Writing that down as a definition in the abstract is easy and empty --- every
configuration of non-negative integer vectors graded by a positive weight would satisfy it. The
content is in the other direction, and it is the question this note ends on:

\begin{quote}
\emph{Which weighted flow configurations arise from a group?} The four regimes above are the image
of one construction; a characterisation of that image would turn the last two rows from a list of
examples into a class, and the thirteen gluing trees of Table~\ref{tab:gluings} are the first
data on how that class behaves.
\end{quote}

This is not a direction nobody has thought of. It is one that was named and left open. Immediately
after the conjecture that the phylogenetic complexity $\varphi(G)$ of an abelian group is at most
$|G|$, \cite{SSphylo} write that the invariant \emph{``makes perfect sense for arbitrary groups not
just abelian groups. However, if $G$ is not abelian then the phylogenetic complexity can exceed the
group order. Using the software 4ti2, we found that $\varphi(S_3,2)\ge8$ for the symmetric group on
three letters. It would be interesting to study the group-theoretic meaning of this invariant.''}
That was 2005. The abelian side of that invariant has moved since --- Micha{\l}ek and Ventura prove
$\varphi(\ZZ_p)$ finite for every prime and $\varphi(\ZZ_3)=3$~\cite{MV}, and the Kimura
3-parameter case is settled in~\cite{MV2} --- while for a non-abelian group the single computed
data point is still the one of 2005.

Two things have to be said in the same breath. First, their $\varphi(G,n)$ is defined for a
\emph{finite} group on a claw tree, while $\Lambda\rtimes\ZZ_m$ is infinite and the alphabet here is
cut out of it by the condition that a character have non-trivial stabiliser: these are not the same
object and nothing below claims they are. Second, what the orbifolds do supply is a supply --- an
infinite family of weighted alphabets with a geometric origin, in which the weight of a label is the
index of the stabiliser of a fixed point, together with their gluing trees, Hilbert series and
degrees. The seven weighted rows of Table~\ref{tab:gluings} are the first entries of that list, and
in all three $\ZZ_6$ cases --- unitary, orthogonal and symplectic --- the gluing degree is $6$,
which is the order of the rotation.

\section{Local and global}\label{sec:localglobal}

One distinction has to be drawn explicitly, because the vocabulary collides. \cite{CFM} prove, under
the heading \emph{complete intersection for claw trees}, that the variety associated with a claw tree
is a complete intersection \emph{in the Zariski open set $U$}, where $U$ meets the variety in its
dense torus orbit. Their Corollary 3.6 says what the equations are: a set of Laurent binomials
defines the variety in $U$ if and only if the corresponding matrices generate the lattice
$\mathrm{adm}(G)$ --- that is, if and only if they are a lattice basis. A lattice-basis ideal has
exactly as many generators as the codimension, for every lattice; the content of their theorem is
the explicit construction and the degree bound, not the property.

The question in this note is the global one: whether the toric ideal itself --- the saturation ---
is generated by codimension-many elements, equivalently whether the semigroup is an iterated gluing.
The two statements are independent, and Table~\ref{tab:tripod} shows they differ: the $\ZZ_4$
tripod is a complete intersection in $U$ by their theorem and is not one globally.

\section{Every $k\ge3$, in one line}\label{sec:allk}

\begin{theorem}[{\cite[Cor.~3.4]{FMS}}]\label{thm:cor34}
Let $S$ be a $d$-dimensional affine semigroup which is a complete intersection, with $d\ge2$. Then
the cone of $S$ has at most $2d-2$ extreme rays.
\end{theorem}

\begin{proposition}\label{prop:allk}
For every $k\ge3$, the semigroup of $(S^1/\ZZ_2)^k$ is not a complete intersection.
\end{proposition}

\begin{proof}
Every generator is an extreme ray. A generator is a $0/1$ vector with exactly one $1$ in each of the
$2^k$ fixed-point blocks, so the generators are \emph{a subset of} the vertex set of the product of
the fixed-point simplices --- a proper subset, since the product orbifold correlates the data at
different fixed points and only $4^k$ of those vertices occur. That is enough: a vertex of a
polytope remains extreme in the convex hull of any subset of the vertices containing it, and all the
generators lie in one weight hyperplane, so each of them spans an extreme ray of the cone. With
$n=4^k$ and $d=3^k$ from~\eqref{eq:kcount}, Theorem~\ref{thm:cor34} requires $4^k\le2\cdot3^k-2$,
and $(4/3)^3>2$.
\end{proof}

\begin{figure}[h]
\centering
\includegraphics[width=0.92\textwidth]{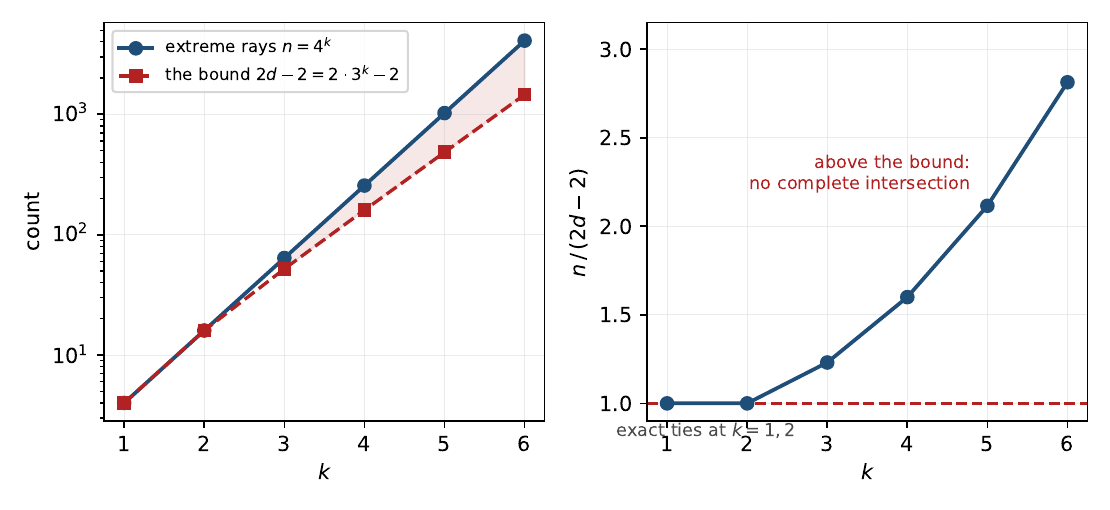}
\caption{Proposition~\ref{prop:allk}, drawn: the extreme rays against the bound, and their ratio.
The ratio is exactly $1$ at $k=1,2$ --- the bound is met, not approached --- and exceeds $1$ from
$k=3$. The ties are why $k=2$ needed the exhaustive sweep, and why no inequality of this shape can
separate a complete intersection from something meeting its bound exactly.}
\label{fig:fms}
\end{figure}

\paragraph{And it is not a cut ideal either.} A cut ideal with $4^k$ coordinates needs a graph on
$n$ vertices with $2^{\,n-1}=4^k$ cuts, so $n=2k+1$, and a marginal space of dimension $3^k$, so
$3^k-1$ edges. A simple graph on $2k+1$ vertices has at most $\binom{2k+1}{2}=k(2k+1)$ edges, so the
requirement is
\begin{equation}\label{eq:edges}
\boxed{\;3^k-1\;>\;\binom{2k+1}{2}=k(2k+1)\qquad(k\ge3),\;}
\end{equation}
which holds at $k=3$ --- $26$ edges wanted on $7$ vertices, which carry at most $21$ --- and
propagates: if $3^k-1>k(2k+1)$ then
\[
3^{k+1}-1\;=\;3(3^k-1)+2\;>\;3k(2k+1)+2\;=\;6k^2+3k+2\;>\;2k^2+5k+3\;=\;(k+1)(2k+3),
\]
the last step because $4k^2-2k-1>0$ for $k\ge1$. So no graph has the right number of edges for
any $k\ge3$. The binary hierarchical description of \S\ref{sec:dictionary} survives; the graph
one does not, which is the precise sense in which $k=2$ is the last case covered
by~\cite{SScuts}.

\begin{ficha}{The three regimes of the product family}
\noindent
The two results above and the exceptional case of \S\ref{sec:exception} are one statement about
$(S^1/\ZZ_2)^k$ read at three values of $k$. Writing $S_k$ for its semigroup, $n=4^k$ and
$d=3^k$ throughout:

\smallskip
\centering\small
\begin{tabular}{lccc}
\toprule
\rowcolor{hdrblue}
& \textcolor{white}{$k=1$} & \textcolor{white}{$k=2$} & \textcolor{white}{$k\ge3$}\\
\midrule
the cut configuration of a graph & yes, path $P_3$ & yes, wheel $W_4$ &
\textbf{no} \ \eqref{eq:edges}\\
\rowcolor{softgrey}
complete intersection & yes, degrees $[2]$ & \textbf{no} & \textbf{no} \ (Prop.~\ref{prop:allk})\\
normal, Cohen--Macaulay & yes & yes~\cite{SScuts} & \emph{open}\\
\rowcolor{softgrey}
Gorenstein & yes (type $1$) & no (type $9$) & ---\\
the framework that still describes it & graph & graph & binary hierarchical\\
\bottomrule
\end{tabular}

\smallskip\raggedright\noindent
The two frontiers do not coincide, and that is the content of the row pair: $k=2$ is still a cut
model and is already not a complete intersection, so the graph description survives one step longer
than the complete-intersection property does. What replaces the graph past $k=2$ is the binary
hierarchical model of the boundary complex of the $k$-dimensional cross-polytope, which is the
$k$-fold join of a pair of points --- and that complex, not the graph, is where the next question
lives.
\end{ficha}

\paragraph{The question this leaves open.} Whether $S_k$ is normal for $k\ge3$ is not settled here.
It matters more than the negative results above: normal affine semigroup rings are Cohen--Macaulay
by Hochster's theorem, so a proof of normality for all $k$ would exhibit an explicit infinite family
of \emph{normal, Cohen--Macaulay, non-complete-intersection} toric rings, indexed by $k$ and coming
from products of orbifolds --- which is a stronger statement than any single ``not a complete
intersection''. The configuration is $A_k=A_1^{\otimes k}$ with $A_1$ the $2\times2$ independence
design, and the two cases in hand are normal, but neither the tensor structure nor the join
structure of the complex is known to us to force it.

\paragraph{The exact ties.} The pair $(n,d)=(2r+2,r+2)$ at $k=1,2$ is the numerology of Example 3.5
of~\cite{FMS}, whose matrix carries two $+1$ and two $-1$ in every row. Reconstructing that example
for $r=1,\dots,4$ returns $2r+2$ extreme rays in dimension $r+2$, meeting the bound exactly, with the
count certified in exact arithmetic. Their example is a complete intersection while $(S^1/\ZZ_2)^2$
is not, with the same $n$, $d$ and rank: so the property is not determined by $(n,d,\operatorname{rank})$,
nor by the coarse sign pattern the two matrices share. What does determine it --- support, circuits,
the oriented matroid --- this pair does not say. The authors say as much about the state of the art~\cite[p.~3144]{FMS}: \emph{``finding
good criteria for establishing that $S$ is not a complete intersection remains an interesting open
problem.''}

\section{The exception, exhaustively}\label{sec:exception}

For completeness, and because it is the case the identification of \S\ref{sec:dictionary} was found
from, here is the direct argument for $k=2$, independent of Table 1 of~\cite{SScuts}.

\begin{proposition}\label{prop:notci}
The semigroup of $S^1/\ZZ_2\times S^1/\ZZ_2$ is not a complete intersection.
\end{proposition}

\begin{proof}
It is not free: its relation lattice has rank $7$. By~\cite{FMS} it is then a complete intersection
if and only if some partition of its minimal generating set into two non-empty parts is a gluing.
There are $2^{15}-1=32767$ such partitions; testing each exactly, $16$ survive the rank condition on
the lattice intersection and none of those has an admissible $\alpha$.
\end{proof}

Two independent measurements agree, and both match~\cite{SScuts}. The minimal Markov basis has $16$
elements --- eight of degree $2$ and eight of degree $4$ --- against a relation lattice of rank $7$.

\paragraph{The $h^*$-vector.} Writing the class count over $(1-x)^9$ gives $(1,7,20,28,7,1)$. That
numerator is an $h^*$-vector and not merely a Hilbert numerator because the semigroup is normal.
Write $L=\ZZ A$ for the group the configuration generates and $L_N$ for its degree-$N$ slice.
Table 1 of~\cite{SScuts} records the cut ring of $\widehat{C_4}=W_4$ as normal and Cohen--Macaulay,
that is $S=\mathrm{cone}(A)\cap L$; since every generator has degree one this gives
$S_N=NP\cap L_N$, so every lattice point of $NP$ in $L$ is a sum of $N$ generators, and the
Hilbert series is the Ehrhart series of $\mathrm{Cut}^\square(W_4)$ with respect to $L$, whose
normalised volume is therefore $64$.

The count generated by the degree-$2$ moves alone gives $(1,7,20,28,15,-7)$ over the same
denominator, and there the argument is negative and needs no such hypothesis: by Stanley's
theorem~\cite{Stanley} a lattice polytope has a non-negative $h^*$-vector, so that series is not the
Ehrhart series of one. It is visible from the shape of the generating function alone, with no
lattice arithmetic.

\paragraph{What the eight quartics are.} Two eights occur in this case and they are the same eight.
The published quadratic relations give $2817$ classes at rank $4$ where the fibres number $2809$, an
excess of $8$; and the minimal Markov basis has $8$ generators of degree $4$. Neither number
explains the other on its face --- a Markov generator is defined by the connectivity of a fibre
under sharing a variable, and the excess is defined by what a particular published list of moves
fails to connect --- so we computed both.

\begin{proposition}\label{prop:eight}
On $S^1/\ZZ_2\times S^1/\ZZ_2$ at rank $4$, exactly eight fibres of the local-datum map are not
connected by the degree-$2$ relations. Each of those fibres contains exactly \emph{two} boundary
conditions, and they are the only fibres of rank $4$ with more than one component. The
corresponding eight quartic binomials are therefore \emph{indispensable}: each occurs, up to
sign, in every Markov basis of the ideal. So are the eight quadrics, for the same reason at rank
$2$; hence
\[
\begin{gathered}
\text{the toric ideal has a \emph{unique} minimal binomial generating set, up to signs:}\\[2pt]
8\ \text{quadrics}\ +\ 8\ \text{quartics.}
\end{gathered}
\]
Under $\mathcal G$ the sixteen fall into exactly two orbits, one per degree.
\end{proposition}

\noindent
Indispensability is what a fibre of exactly two disconnected monomials gives: if $F_b=\{u,v\}$ and
no binomial of lower degree connects them, then $x^u-x^v$ has nowhere else to come from, so every
Markov basis contains it up to sign. That is strictly stronger than being a member of \emph{some}
minimal Markov basis, and it is what turns ``a minimal Markov basis has $8+8$'' into ``the minimal
Markov basis is $8+8$''. The control that it is not simply ``the fibre has two elements'' is that
fibres of ranks $3$ and $4$ with several monomials, all connected, exist and correctly ask for
nothing.

\noindent
So the first failure of the quadratic relations is one phenomenon seen eight times: a single orbit
of \emph{pairs} of boundary conditions with equal local data that no quadratic move exchanges, one
quartic apiece. Up to symmetry the whole ideal is two relations, one of each degree, and the
second exists because of those pairs and for no other reason.

\paragraph{The group, exactly.} $\mathcal G\le S_{16}$ is generated by seven involutions on the
labels: the four \emph{switchings}, each flipping the parity under one of the four reflections;
the two swaps exchanging the two reflections within a circle; and the swap exchanging the two
circles. It is $\ZZ_2^4\rtimes\mathrm{Aut}(W_4)$ --- switching by a vertex subset, semidirect
the automorphisms of the wheel --- of order $16\cdot8=128$. It is not merely a subgroup: computing
the permutations of the sixteen labels that preserve the relation lattice $\ker A$ directly, by a
route that does not know the construction above, returns exactly those $128$. So
$\mathcal G=\mathrm{Aut}(A)$, and ``a single orbit'' is a statement about the full automorphism
group.

The same group governs the canonical module: its eight minimal generators of degree $6$
(Proposition~\ref{prop:type}) are one $\mathcal G$-orbit as well, so the whole of the exceptional
structure --- relations and canonical module both --- is carried by orbits of size eight and one.

\begin{table}[h]
\centering\small
\begin{tabular}{ll l}
\toprule
\rowcolor{hdrblue}
\textcolor{white}{invariant} & \textcolor{white}{value} & \textcolor{white}{source}\\
\midrule
vertices, dimension, facets & $16$, \ $8$, \ $24$ & computed here\\
\rowcolor{softgrey}
$f$-vector & $(1,16,104,360,712,816,520,168,24,1)$ & computed here\\
$h^*$-vector & $(1,7,20,28,7,1)$ & computed here\\
\rowcolor{softgrey}
normalised volume & $64$ & $\sum h^*_i$; also the degree in~\cite{SScuts}\\
minimal generators & $8$ quadrics, $8$ quartics & computed here; \cite{SScuts}, Table 1\\
\rowcolor{softgrey}
normal, Cohen--Macaulay & yes, yes & \cite{SScuts}, Table 1\\
Gorenstein & no --- Cohen--Macaulay \textbf{type} $9$ & computed here\\
\rowcolor{softgrey}
$\operatorname{codeg}P$, \ $a(k[S])$, \ $\operatorname{reg}k[S]$ & $4$, \ $-4$, \ $5$ &
computed here\\
complete intersection & \textbf{no} & Proposition~\ref{prop:notci}\\
\bottomrule
\end{tabular}
\caption{The exceptional case, as a list of invariants. It is a normal, Cohen--Macaulay toric ring
which is not a complete intersection --- which is the useful way to say what it is, rather than
saying only what it is not.}
\label{tab:invariants}
\end{table}

\paragraph{How far from Gorenstein.} ``Not Gorenstein'' is the least informative thing one can
say about a Cohen--Macaulay ring, because Gorenstein is the case $\mathrm{type}=1$ of a positive
integer. Since the semigroup is normal, Danilov--Stanley identifies the canonical module
$\omega$ with the ideal spanned by the interior lattice points of the cone, and the type is its
minimal number of generators.

\begin{proposition}\label{prop:type}
The semigroup ring of $S^1/\ZZ_2\times S^1/\ZZ_2$ has Cohen--Macaulay type $9$. Its canonical
module has one minimal generator in degree $4=\operatorname{codeg}P$ and eight in degree $6$, and
none in any other degree.
\end{proposition}

\noindent
The criterion is exact and finite: $u\in\operatorname{relint}(S)$ is a minimal generator of
$\omega$ when $u-s\notin\operatorname{relint}(S)$ for every non-zero $s\in S$, and since $S$ is
generated in degree one it suffices to test $u-a_i$ over the sixteen generators. Interior
membership is decided against the $24$ facets of the cone, computed in exact rational arithmetic;
and because the semigroup is normal, its degree-$N$ part \emph{is} the set of lattice points of
$NP$, so those can be enumerated as sums of $N$ generators with no integer program anywhere.

Two counts appear here and they are different things. The interior points of the cone number
$1,16,136,752,3097$ in degrees $4$ to $8$; the minimal generators of $\omega$ number $1,0,8,0,0$.
There is no tension: the sixteen interior points of degree $5$ are the single degree-$4$ generator
plus each of the sixteen generators of $S$, so they are generated and not minimal. The type counts
only what addition by $S$ cannot reach. The interior counts are produced twice --- once by
enumeration and once as the coefficients of $(-1)^{d}H(1/x)$ --- and the two agree, so the
enumeration is of the right module. The asymmetry of $(1,7,20,28,7,1)$ says only that the type is
not $1$; it does not say that it is $9$.

\begin{figure}[h]
\centering
\includegraphics[width=0.78\textwidth]{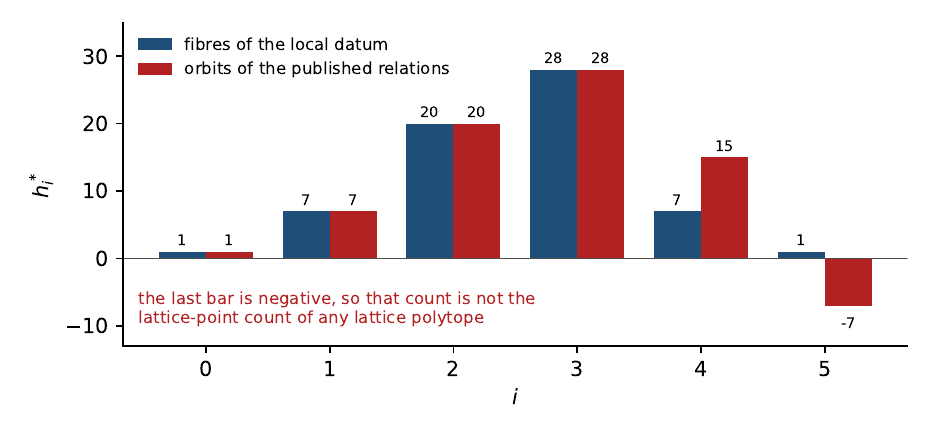}
\caption{The two counts of $S^1/\ZZ_2\times S^1/\ZZ_2$ written over $(1-x)^9$. The first four
coefficients agree; the last is negative on the right, and a lattice polytope cannot produce a
negative $h^*_i$. The two series are $1,16,128,688,\mathbf{2809}$ and
$1,16,128,688,\mathbf{2817}$.}
\label{fig:hstar}
\end{figure}

\section{The marginal polytopes}\label{sec:polytopes}

For the weight-one alphabets the generators lie in a hyperplane and their convex hull is a lattice
polytope of dimension $\dim S-1$, with every generator a vertex.

\begin{table}[h]
\centering\small
\begin{tabular}{lrrl}
\toprule
\rowcolor{hdrblue}
\textcolor{white}{orbifold} & \textcolor{white}{vertices} & \textcolor{white}{$\dim$} &
\textcolor{white}{combinatorial type}\\
\midrule
$S^1/\ZZ_2$ & $4$ & $2$ & $\Delta_1\oplus\Delta_1$\\
\rowcolor{softgrey}
$T^2/\ZZ_2$ & $8$ & $4$ & $\Delta_1^{\oplus4}$; also $\mathrm{Cut}^\square(C_4)$, as it must be\\
$T^2/\ZZ_3$ & $9$ & $6$ & $\Delta_2\oplus\Delta_2\oplus\Delta_2$\\
\rowcolor{softgrey}
$T^2/\ZZ_4$ (weight $1$) & $8$ & $6$ & $\Delta_3\oplus\Delta_3$\\
$T^2/\ZZ_6$ (weight $1$) & $6$ & $5$ & $\Delta_5$\\
\midrule
\rowcolor{softgrey}
$S^1/\ZZ_2\times S^1/\ZZ_2$ & $16$ & $8$ & $\mathrm{Cut}^\square(W_4)$\\
\bottomrule
\end{tabular}
\caption{The marginal polytopes. Above the line the type is a free sum of simplices; below it is a
cut polytope of a graph with a $K_4$ minor.}
\label{tab:polytopes}
\end{table}

The free sums are the polytope shadow of Table~\ref{tab:gluings}, and the correspondence is exact on
these five rows: the number of gluings is the number of free summands minus one, and each gluing
degree is the number of vertices of the summand it splits off. The reading is specific to the
unit-weight slice, and not because it fails elsewhere: once the alphabet carries weights the
generators no longer lie in one degree-one hyperplane and the convex hull is not the object to
compare with. The first weighted case, $T^2/\ZZ_3$ over $\so N$ with its gluing degree $2m$, is
where a naive extension would already have to be reconsidered.

\section{Ledger}\label{sec:ledger}
{\small
\begin{longtable}{p{6.2cm}p{4.2cm}p{4.2cm}}
\caption{What each result owes.}\label{tab:ledger}\\
\toprule
\rowcolor{hdrblue}
\textcolor{white}{Item} & \textcolor{white}{Inherited (cite)} & \textcolor{white}{Ours}\\
\midrule
\endfirsthead
\toprule
\rowcolor{hdrblue}
\textcolor{white}{Item} & \textcolor{white}{Inherited (cite)} & \textcolor{white}{Ours}\\
\midrule
\endhead

complete intersection $\iff$ iterated gluing; the extreme-ray bound &
Fischer--Morris--Shapiro~\cite{FMS}; Delorme~\cite{Delorme} in rank one & nothing\\
\midrule
\rowcolor{softgrey}
Hilbert series of a gluing~\eqref{eq:ago} & Assi--Garc\'ia-S\'anchez--Ojeda~\cite{AGO} & nothing\\
\midrule
cut ideals; the invariants of the cut ring of $W_4$ --- codimension, degree, minimal generators;
quadrics $\iff$ no $K_4$ minor &
Sturmfels--Sullivant~\cite{SScuts}; Engstr\"om~\cite{Engstrom} & nothing\\
\midrule
\rowcolor{softgrey}
normality and Cohen--Macaulayness of the exceptional case, on which the reading of the numerator as
an $h^*$-vector rests & Sturmfels--Sullivant~\cite{SScuts}, Table 1 & nothing\\
\midrule
group-based models, their codimension, and that most are not complete intersections &
Sturmfels--Sullivant~\cite{SSphylo}; Casanellas--Fern\'andez-S\'anchez--Micha{\l}ek~\cite{CFM} &
nothing\\
\midrule
\rowcolor{softgrey}
non-negativity of $h^*$ & Stanley~\cite{Stanley} & nothing\\
\midrule
\textbf{that orbifold boundary conditions form these semigroups, and which graph, tree or complex
each one is} & --- & \textbf{ours}\\
\midrule
\rowcolor{softgrey}
\textbf{that each weighted alphabet is the minimal generating set} & --- & \textbf{ours}\\
\midrule
\textbf{the gluing trees of the weighted rows, and of the orthogonal and symplectic alphabets} &
--- & \textbf{ours}\\
\midrule
\rowcolor{softgrey}
\textbf{that the tripod is a complete intersection exactly for $|G|\le3$, for every finite abelian
$G$} & --- & \textbf{ours}\\
\midrule
\textbf{that $(S^1/\ZZ_2)^k$ is not a complete intersection for any $k\ge3$, and is not a cut
ideal} & --- & \textbf{ours}\\
\midrule
\rowcolor{softgrey}
\textbf{the Cohen--Macaulay type, $a$-invariant and codegree of the exceptional case, and that
its minimal Markov basis is unique up to signs, sixteen indispensable binomials in two
orbits of eight under $\mathrm{Aut}(A)$} &
that it is normal and Cohen--Macaulay and not Gorenstein: \cite{SScuts}, Table 1 &
\textbf{ours} (the type $9$ itself, and the degrees its canonical module is generated in)\\
\bottomrule
\end{longtable}}

\appendix
\section{Reproducibility}\label{app:repro}
Every number regenerates from the ancillary scripts, each of which writes its own receipt and states
its own controls.

{\small
\begin{longtable}{p{5.4cm}p{9.4cm}}
\toprule
\rowcolor{hdrblue}
\textcolor{white}{script} & \textcolor{white}{what it establishes}\\
\midrule
\endfirsthead
\toprule
\rowcolor{hdrblue}
\textcolor{white}{script} & \textcolor{white}{what it establishes}\\
\midrule
\endhead

\texttt{bc\_preflight.py} & the alphabets, weights and local data, from the rotation matrix\\
\rowcolor{softgrey}
\texttt{cut\_ideal\_gate.py} & that the labels are the cuts of $P_3$, $C_4$ and $W_4$, both
directions\\
\texttt{gluing\_tree.py} & the ten unitary and orthogonal gluing trees; the exhaustive sweep of
the $32767$ partitions\\
\rowcolor{softgrey}
\texttt{tripod\_gate.py} & the tripod for every finite abelian group, with the codimension control\\
\texttt{sp\_gluing\_gate.py} & the three symplectic gluing trees: minimality, gluings, degrees\\
\rowcolor{softgrey}
\texttt{ledger\_gate.py} & that every script this table names is on disk with its receipt,
read out of this note's own source; and that $10+3=13$, with every row of
Table~\ref{tab:gluings} backed by exactly one of the two receipts and every receipt row used by
exactly one table row\\
\texttt{gorenstein\_gate.py} & the Cohen--Macaulay type, the $a$-invariant and $\operatorname{codeg}$
of the exceptional case, with a Gorenstein decoy\\
\rowcolor{softgrey}
\texttt{quartic\_orbit\_gate.py} & the eight split fibres, the eight degree-$4$ generators, and
that they are the same eight\\
\texttt{indispensable\_gate.py} & Proposition~\ref{prop:eight}: that all sixteen are
indispensable, that $\mathrm{Aut}(A)$ has order $128$, and the two binomial orbits together with
the one of the canonical generators\\
\rowcolor{softgrey}
\texttt{complete\_intersection.py} & minimal Markov basis against lattice rank\\
\texttt{fms\_cone\_gate.py} & extreme rays, $n=4^k$, $d=3^k$, and the control on Example 3.5\\
\rowcolor{softgrey}
\texttt{formula\_atlas.py} & the Hilbert series and the two $h^*$-vectors\\
\texttt{polytope\_id.py} & the combinatorial types of the marginal polytopes\\
\rowcolor{softgrey}
\texttt{minimality\_gate.py} & that every listed alphabet is the minimal generating set, with a decoy\\
\bottomrule
\end{longtable}}

\noindent
One receipt in \texttt{outputs/} has no row above, and deliberately: \texttt{figure\_table\_gate.txt}
is the run of a gate that reads \emph{both} notes' sources and checks that every sequence a figure
is drawn from is the sequence its table prints --- for this note, the two curves of
Figure~\ref{fig:fms} and the two bar sets of Figure~\ref{fig:hstar}. Its script needs the companion
note's source as well, so it travels with~\cite{IXA}; the receipt is carried here because it is this
note's figures that it certifies.

\medskip\noindent
Two controls are worth naming because they failed first and were repaired. The extreme-ray count
must be run on a semigroup and not on an arbitrary basis of a relation lattice, since an arbitrary
basis gives a cone containing a line; and a decoy must be a decoy --- the first cone written to test
the counter had two of its five generators inside the cone.



\begin{thebibliography}{99}

\bibitem{IXA} C.~Mar\'in, \emph{The Alphabet of Orbifold Boundary Conditions}, Part IX-A of this
series, Zenodo, \texttt{doi:10.5281/zenodo.22254861}.

\bibitem{SScuts} B.~Sturmfels, S.~Sullivant, \emph{Toric geometry of cuts and splits}, Michigan
Math.\ J.\ \textbf{57} (2008) 689--709, arXiv:math/0606683. 

\bibitem{Engstrom} A.~Engstr\"om, \emph{Cut ideals of $K_4$-minor free graphs are generated by
quadrics}, Michigan Math.\ J.\ \textbf{60} (2011) 705--714, arXiv:0805.1762. 

\bibitem{SSphylo} B.~Sturmfels, S.~Sullivant, \emph{Toric ideals of phylogenetic invariants},
J.\ Comput.\ Biol.\ \textbf{12} (2005) 204--228, arXiv:q-bio/0402015. 

\bibitem{CFM} M.~Casanellas, J.~Fern\'andez-S\'anchez, M.~Micha{\l}ek, \emph{Local description of
phylogenetic group-based models}, arXiv:1402.6945. 

\bibitem{MV} M.~Micha{\l}ek, E.~Ventura, \emph{Finite phylogenetic complexity of $\ZZ_p$ and
invariants for $\ZZ_3$}, J.\ Algebra (2017), arXiv:1508.04177.

\bibitem{MV2} M.~Micha{\l}ek, E.~Ventura, \emph{Phylogenetic complexity of the Kimura
3-parameter model}, arXiv:1704.02584.

\bibitem{FMS} K.~G.~Fischer, W.~Morris, J.~Shapiro, \emph{Affine semigroup rings that are complete
intersections}, Proc.\ Amer.\ Math.\ Soc.\ \textbf{125} (1997) 3137--3145. 

\bibitem{Delorme} C.~Delorme, \emph{Sous-mono\"\i des d'intersection compl\`ete de $N$}, Ann.\ Sci.\
\'Ec.\ Norm.\ Sup\'er.\ \textbf{9} (1976) 145--154. 

\bibitem{AGO} A.~Assi, P.~A.~Garc\'ia-S\'anchez, I.~Ojeda, \emph{Frobenius vectors, Hilbert series
and gluings of affine semigroups}, J.\ Commut.\ Algebra \textbf{7} (2015) 317--335. 

\bibitem{Stanley} R.~P.~Stanley, \emph{Decompositions of rational convex polytopes}, Ann.\ Discrete
Math.\ \textbf{6} (1980) 333--342.

\bibitem{HHK} N.~Haba, Y.~Hosotani, Y.~Kawamura, Prog.\ Theor.\ Phys.\ \textbf{111} (2004) 265.

\bibitem{KM} Y.~Kawamura, T.~Miura, Prog.\ Theor.\ Phys.\ \textbf{122} (2009) 847, arXiv:0905.4123.

\bibitem{TI2} K.~Takeuchi, T.~Inagaki, \emph{Classification of $T^2/\ZZ_m$ orbifold boundary
conditions}, arXiv:2404.19411. 

\bibitem{TI3} K.~Takeuchi, T.~Inagaki, PTEP \textbf{2025} 043B03, arXiv:2501.05849. 

\end{thebibliography}
\end{document}